\documentclass[letterpaper, 10pt,journal]{IEEEtran}

\IEEEoverridecommandlockouts                              % This command is only
\usepackage{graphicx}
\usepackage{amsmath}
\usepackage{cite}
\usepackage{mathtools, amssymb}
\usepackage{varioref}
\labelformat{equation}{(#1)}
\usepackage{amsfonts}

\usepackage{bbm,dsfont}
\usepackage{color}
\usepackage{enumerate}
\usepackage[hidelinks]{hyperref}
\hypersetup{pdfpagemode=UseNone}
\usepackage{empheq}

\usepackage{mysymbol}

\usepackage[dvipsnames]{xcolor}

\usepackage{tikz,pgfplots}
\usepackage{color} % change text color        
\usetikzlibrary{matrix} % for block alignment
\usetikzlibrary{arrows} % for arrow heads
\usetikzlibrary{calc} % for manimulation of coordinates
\usetikzlibrary{shapes}
\usepackage{standalone}
\pgfplotsset{compat=1.6}

\usetikzlibrary{decorations.pathreplacing}

\usepackage{amsthm}
\theoremstyle{plain}
\newtheorem{thm}{Theorem}

\theoremstyle{definition}
\newtheorem{asmptn}{Assumption}
\newtheorem{rem}{Remark}

\newtheorem{example}{Example}

\renewcommand{\epsilon}{\varepsilon}

\usepackage{epstopdf}

\usepackage{accents}
\newcommand{\ubar}[1]{\underaccent{\bar}{#1}}

\title{%\LARGE \bf
Latency-Reliability Tradeoffs for State Estimation}

\author{Konstantinos Gatsis, Hamed Hassani, George J. Pappas
\thanks{
 This research is partially supported by NSF CPS-1837253, NSF-CCF 1755707, and by the Intel Science and Technology Center for Wireless Autonomous Systems (ISTC-WAS)
	and by ARL DCIST CRA W911NF-17-2-0181. The authors are with the Department of Electrical and Systems Engineering, University of Pennsylvania, 200 South 33rd Street, Philadelphia, PA 19104. Email: \{kgatsis, hassani, pappasg\}@seas.upenn.edu.}%
}

\begin{document}

\maketitle

\begin{abstract}
The emerging interest in low-latency high-reliability applications, such as connected vehicles, necessitates a new abstraction between communication and control.
Thanks to advances in cyber-physical systems over the past decades, we understand this interface for classical bit-rate models of channels as well as packet-loss-type channels.  
This work proposes a new abstraction characterized as a tradeoff curve between latency, reliability and rate. Our aim is to understand: Do we (control engineers) prefer faster but less reliable communications (with shorter codes), or slower but more reliable communications (with longer codes)?  
In this paper we examine the tradeoffs between latency and reliability for the problem of estimating dynamical systems over communication channels. Employing different latency-reliability curves derived from practical coding schemes, we develop a co-design methodology, i.e., select the code length depending on the system dynamics to optimize system performance. 
\end{abstract}

%%% I added these to force page numbers
\thispagestyle{plain}
\pagestyle{plain}
%%% Replace them with the following:
%\thispagestyle{empty}
%\pagestyle{empty}

\section{Introduction}\label{sec:intro}

Recent interest in the Internet-of-Things (IoT) and the next generation wireless communication standards (5G) is targeting applications such as connected vehicles, collaborative swarm planning, smart cities, and industrial control~\cite{agiwal2016next}. These are challenging applications due to their low-latency high-reliability closed-loop control requirements. This necessitates \emph{rethinking} the communication stack, practical codes, networking architecture, and control design that can provide ultra low latency ($<$1ms) and very high reliability ($>$99.999\%), which is not possible in today's wireless systems. Even more fundamentally, we lack an understanding of the limits for stability, estimation, and control over low-latency, high-reliability communications.

In this paper, we take the first steps in understanding the fundamental tradeoffs between latency, reliability, stability, and control performance. More simply, we ask:\\
{\em Does a dynamical system need faster but less reliable information or slower but more reliable information?} 

To answer this question we propose a new communication abstraction (Figure~\ref{fig:abstraction}) based on the latency-rate-reliability curves obtained from recent developments in information/coding theory for finite blocklengths. Furthermore, we argue fundamentally about the opportunity of \textit{co-designing} the dynamical (control) system and communication block and the role the key parameters (i.e. rate, reliability, error, system dynamics) play in such a co-design. 

One of the earliest abstractions between communication and control is based on bit rate characterizations. This has permitted a fundamental understanding of the minimum bit rate required for stability over data-rate limited channels~\cite{Tatikonda, Nair_overview, franceschetti2014elements, Liberzon_Mitra_2018} as well as for the case where besides data rate constraints channels also introduce noise~\cite{sahai2006necessity, tatikonda2004control}. Beyond fundamental characterizations this has led to an extensive literature on appropriate quantizer, encoder, and decoder designs for control~\cite{brockett2000quantized, Tatikonda, yuksel2006minimum, ostrovsky2009error, sukhavasi2016linear,kostina2016rate,khina2017algorithms,varaiya1983causal, borkar1997lqg,khina2016multi,kostina2018exact}. 
A different widely adopted abstraction is packet-based communication where quantization and data rate effects are often ignored. This has permitted the analysis of the maximum packet drop rate above which controlling a plant becomes impossible~\cite{Sinopoli_intermittent, Basar_optimal_control_over_links, Hespanha_survey, Gupta_LQG}, as well as observer and controller design to counteract random packet drops~\cite{ Walsh_stability,  smith2003estimation, Basar_optimal_control_over_links, Hespanha_survey, Schenato_foundations, Gupta_LQG}. 

\begin{figure}
	\centering
	\includegraphics[width = 0.8\columnwidth]{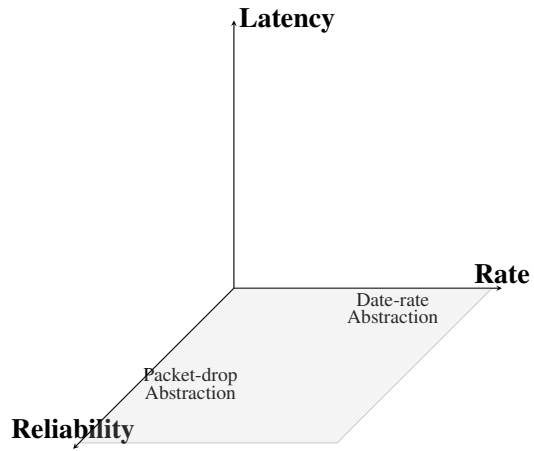}
	\caption{Rate-Latency-Reliability Design Space for Networked Estimation and Control. Existing works explore the design space spanned by rate and reliability axes. This paper introduces a methodology for exploring latency as well as its tradeoffs with rate and reliability from the perspective of state estimation.}
	\label{fig:abstraction}
\end{figure}

Finally the recent interest in low-power IoT devices has increased the development of abstractions in terms of resources, e.g., the number of transmissions, and has resulted in the development of frameworks such as event-triggered control \cite{Event_triggered_intro, Hespanha_main, lipsa2011remote, imer2010optimal,Hirche_joint,nayyar2013optimal} and transmit power allocation~\cite{GatsisEtal14,Quevedo_power_coding}. 
Resource allocation in more complex networks with multiple plants, sensors, or actuators has also attracted attention but primarily without latency considerations. Examples include sharing a communication medium between different sensor and actuators~\cite{ Branicky_stability, heemels2010networked,Donkers_switched,LeNy_resource_LQR, Hirche_Scheduling_Price}, decentralized mechanisms subject to interferences~\cite{Johansson_dual_TAC,Martins_estimation_shared}, and control over shared wireless channels~\cite{GatsisEtal15, GatsisEtal18}. {Recent approaches also consider control system operation over multiple links where each link has a different given delay parameter~\cite{smarra2018efficient,maity2018optimal}.}

{On the other hand, low-latency communication introduces a novel design parameter that is absent in the above literature. In particular, delay has been treated as a given disturbance to a control system either fixed or random. Longer delays are undesirable because, as expected, they degrade system performance, but they have not been explicitly part of the design.
{In that sense the existing literature explores the design space spanned by the reliability and data rate axes of Figure~\ref{fig:abstraction} from the perspective of control systems.} 
Our novelty is on analyzing the effect of latency in performance and appropriately guiding the selection of code blocklength.}

We note that interest in low-latency communications has also started to appear within the networking community, using the concept of age or freshness of information~\cite{sun2017update,he2018optimal,kadota2018optimizing},
or analyzing latency and reliability using network coding and cooperative communication~\cite{swamy2017real, swamy2016network, dickstein2016finite, hu2018finite}. 

\begin{figure}
	\centering
	\includegraphics[width = \columnwidth]{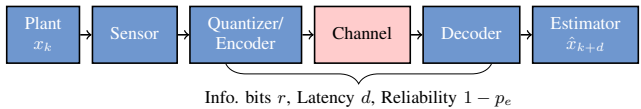}
	\caption{{System for Low-latency High-Reliability Remote Estimation. The communication block consisting of encoder, channel, and decoder is abstracted as a latency-reliability curve.}}
	\label{fig:channel_model}
\end{figure}

In information theory, fundamental limits for low-latency communication have been characterized in \cite{polyanskiy2010channel,dobrushin1961mathematical, strassen}  (also known as ``non-asymptotic channel coding", see Section~\ref{sec:coding} and \ref{opt_scaling} for details). Such laws demonstrate how the minimum communication (coding) latency scales with rate and reliability. Such laws are typically used as a benchmark for comparing the latency of a coding scheme with the optimal latency. Even though it has been shown that codes with optimal latency exist,  so far none of the practical coding schemes (such as polar or iterative codes) have been capable of achieving such optimal latencies \cite{hassani2014finite, mondelli2016unified,mondelli2014achieve}. Hence, designing practical codes with optimal latency is a key research frontier in information theory.
In this paper instead we use such fundamental scaling laws from information theory to abstract the communication block (encoding/channel/decoding) of the system shown in Figure~\ref{fig:channel_model} that includes the physical plant dynamics. In other words, we can find out what the optimal (shortest) range of the transmission length (coding latency) should be from the perspective of dynamical system performance. Such an abstraction of the communication block can then help us to provide fundamental tradeoffs between reliability and latency for the whole system shown in Figure~\ref{fig:channel_model}. 
Even though we do not derive any new blocklength-reliability curves, to the best of our knowledge we are the first to use them to optimize estimation performance for dynamical systems.

In Section~\ref{sec:setup} we describe our setup which involves the remote estimation of a scalar dynamical system over a communication channel. We introduce our abstraction capturing the available information bits, latency, and reliability of the communication block. Using a sequential quantization scheme from the control literature~\cite{Tatikonda} we analyze estimation performance. In particular we show that reliability plays a crucial role in keeping the estimation stable (cf.~Theorem~\ref{thm:stability}) while both latency and reliability need to be taken into account to optimize the steady state estimation performance (cf.~Theorem~\ref{thm:error_bound}).

We proceed in Section~\ref{sec:coding} to characterize the role of coding in the proposed abstraction. We utilize known blocklength-reliability curves of both theoretically optimal and practical coding schemes and we consider their effect on the system performance, i.e., the estimation error, exploiting the control-theoretic results of Section~\ref{sec:setup}. {Longer codes, even though they improve reliability, should be avoided due to their latency impact on performance. On the other hand, too short codes do not have the necessary reliability to even stabilize the system. Our contribution is a methodology that facilitates the choice of optimal code length and reveals its relationship with system dynamics and channel conditions. 
 To the best of our knowledge this is the first time such a fundamental relationship is derived. }This analysis is extended to higher-dimensional systems in Section~\ref{sec:multidimensional}.

{In summary, the contributions of our paper are:
\begin{itemize}
	\item a new communication abstraction that includes rate, latency, and reliability as design parameters based on finite blocklength codes
	\item an analysis of state estimation performance over this abstraction (Theorem~\ref{thm:error_bound})
	\item a methodology for selection of the optimal coding blocklength for the first time for the problem of state estimation. 
\end{itemize}}

\section{Problem Setup}\label{sec:setup}

Consider the setup indicated in Figure~\ref{fig:channel_model}. A sensor is measuring the state of a dynamical process and a remote estimator is interested in maintaining a state estimate of the dynamical process. This is achieved by \emph{communication} between the two entities over a \emph{noisy channel}. Specifically, we consider the discrete time scalar dynamical system - the general case of higher-dimensional systems is treated in Section~\ref{sec:multidimensional}
\begin{equation}\label{eq:system}
	x_{k+1} = A x_k + w_k
\end{equation}
where $k=0,1,\ldots$ are the discrete time steps, $x_k \in \reals$ is the state of the dynamical system where the initial state  $x_0$ is known to lie in some set $[\ubar{X}_0, \bar{X}_0] \subset \mathbb{R}$%[-X_0/2, X_0/2]
, and $w_k$ is an unknown disturbance of magnitude $w_k \in [-W/2, W/2]$. This assumption of \emph{bounded} initial condition and disturbance is common in the literature, see, e.g., \cite{Tatikonda,sahai2006necessity,sukhavasi2016linear} -- see also available approaches for the case of unbounded stochastic disturbances~\cite{Nair_overview, kostina2018exact}. {Equation \ref{eq:system} is derived from discretizing a continuous time dynamical system over small units of time. Without loss of generality, as we will also make clear in Section~\ref{sec:coding} we may take each discrete time step normalized to correspond to the time interval required for the transmission of a single bit over the channel. Then we may define all other time intervals as multiples of these discrete units of time.
}

We let $A\geq 0$ and note that similar analysis can be given for the symmetric case $A\leq 0$. In general this model captures the important case $A>1$, i.e., where the system is unstable and the remote estimator is interested in tracking the state with a bounded error even though the state can grow unbounded. 

The sensor samples the dynamical system \ref{eq:system} every $T$ discrete time steps, i.e., at times $k=0, T, 2T, \ldots$. Thus, from the perspective of the sensor this is equivalent to sampling a dynamical system of the form
\begin{equation}
x_{(\ell+1) T} = A^T x_{\ell T} + \sum_{j=0}^{T-1} A^{T-1-j} w_{\ell T +j}
\end{equation}
%
%with a disturbance of magnitude $w'_k = \sum_{j=0}^{T-1} A^{T-1-j} w_{k+j} \in [-\frac{A^T-1}{A-1} W/2, \frac{A^T-1}{A-1} W/2]$
with the index $\ell =0, 1, \ldots$ counting the number of generated samples. 
Throughout the paper this sampling period $T$ is fixed and not a design parameter. 
{For notational convenience we represent the system dynamics in the form \ref{eq:system} with respect to the (shorter) discrete time steps corresponding to  transmission intervals of single bits, while the (longer) sampling period is a multiple of these discrete time steps -- see also Figure~\ref{fig:timing}.}

At the sensor each sample is converted (quantized) to an $r$-bit message. The message to be communicated is transformed into a generally longer $n$-bit message by using some form of channel coding procedure-- see details in Section~\ref{sec:coding}. The transmission of the coded message to the remote estimator introduces a \emph{delay}\footnote{{We note that a different notion of delay is discussed in \cite{sahai2006necessity,sukhavasi2016linear}. In particular this has to do with the case where communication is over a noisy channel and it takes a random number of channel uses for the decoder to correctly decode past transmitted messages. This is different than transmission latency in our work}}. In particular in this paper we account for the time to transmit each bit over the channel. We note that in practice we can extend our model to include other message overheads that introduce delays, as well as the process of encoding and decoding which introduces computational delay, or there might be uncertainty in the delay. Additionally communication introduces noise. We make the following assumptions for the delay, noise, and availability of channel feedback. As we will see, these assumptions are fairly general depending on how the parameters are set and cover a variety of practical scenarios. 

\begin{asmptn}\label{as:model}
	Each $r$-bit message requires a delay of $d$ time steps, with $d\leq T$. Each message is either successfully decoded with probability $1-p_e$, or with probability $p_e$ it is corrupted during the communication and discarded. 
\end{asmptn}

\begin{asmptn}\label{as:ack}
	The receiver/remote estimator sends perfect acknowledgment signals to the transmitter about whether each message is successfully received or not.
\end{asmptn}

\begin{figure}
	\centering
	\includegraphics[width = \columnwidth]{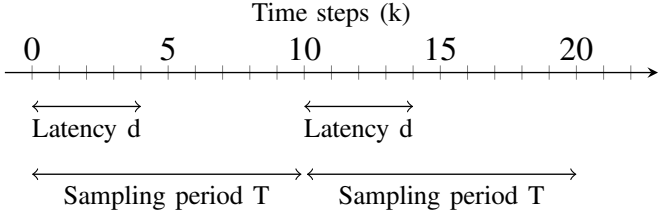}
	\caption{{Time steps indexed by $k=0, 1, \ldots$ correspond to time intervals required for the transmission of single bits over the channel. The sensor samples and transmits the state of the plant periodically and this period corresponds to $T$ time steps, e.g., $T=10$ above. The transmission of each sample which consists of multiple bits requires $d$ time steps, e.g., $d=4$ above.}}
	\label{fig:timing}
\end{figure}

Assumption~\ref{as:model} describes our proposed \textit{abstraction for low-latency high-reliability dynamical systems}. As shown in Figure~\ref{fig:channel_model} this abstraction by the length of the message $r$, the latency $d$, and the reliability $1- p_e$ corresponds to the whole \textit{communication block} including the encoding, channel, and decoding. In this section we analyze the estimation performance treating these as given parameters. In practice, as we discuss in Section~\ref{sec:coding}, these are interrelated parameters as derived from practical coding schemes and cross-layer optimization. For example keeping the channel reliability $1 - p_e$ fixed, we can increase the quantization level $r$ to get more precise information, but this will increase the delay $d$ of the communication which may adversely affect the estimation performance. Similar tradeoffs arise by tuning any of the parameters. {We also note that the assumption that latency is less than the sampling period ($d\leq T$) is added to practically ensure that messages are received before new messages are being generated to be sent over the communication block.}

The fact that the receiver knows whether the message is successfully decoded or not (Assumption~\ref{as:model}) is an important assumption and is often made in the control literature \cite{Tatikonda,tatikonda2004control}-- see also Remark~\ref{rem:detection}. Assumption~\ref{as:ack} is also typical.

To sum up, the state is sampled at time steps $0, T, 2T, \ldots$ and transmission takes $d$ time steps -- see also Figure~\ref{fig:timing}. {Similar models are common in the literature~\cite{Donkers_switched}.} At time steps $d, T+d, 2T+d, \ldots$ the remote estimator builds some estimates $\hat{x}_d, \hat{x}_{T+d}, \hat{x}_{2T+d}, \ldots$ about the corresponding state and we measure performance by the magnitudes of the estimation error $|x_d - \hat{x}_d|, |x_{T+d} - \hat{x}_{T+d}|, |x_{2T+d} - \hat{x}_{2T+d}|, \ldots$. Other performance metrics are also possible.

%%%%%%%%%%%%%%%%%%%%%%%%%%%%%%%%%%%%%%
%%%%%%%%%%%%%%%%%%%%%%%%%%%%%%%%%%%%%%
\subsection{Quantization scheme}
%%%%%%%%%%%%%%%%%%%%%%%%%%%%%%%%%%%%%%
%%%%%%%%%%%%%%%%%%%%%%%%%%%%%%%%%%%%%%

The way each state sample is quantized to an $r$-bit message is an important choice. Here we employ the scheme of Tatikonda-Mitter~\cite{Tatikonda} which is a sequential quantization scheme. 
We also point out that other quantization schemes are also available in the literature, e.g.,~ \cite{Nair_overview,sukhavasi2016linear,kostina2018exact}, but not explored in this work.

Our coding scheme is as follows. At time $k=0$ the receiver/estimator knows that the state belongs in the initial set $[\ubar{X}_0, \bar{X}_0]$. The transmitter uniformly quantizes this range of values into $2^r$ bins and transmits the $r$-bit message indicating the bin in which the measured state $x_0$ belongs. After the transmission duration of $d$ time steps we have two cases.

\underline{Case I:} If the message corresponding to the sample $x_0$ is not successfully received at time $k=d$, the remote estimator only knows the initial information, i.e., that $x_0 \in [\ubar{X}_0, \bar{X}_0]$. Then it can form an estimate about $x_0$ as the center value of this interval
\begin{equation}\label{eq:coding_start}
	\hat{x}_0 = \frac{\bar{X}_0 + \ubar{X}_0 }{2} .
\end{equation}
However the estimator is interested in the \textit{current} value of the state $x_d$ at this reception time $k=d$. To obtain an estimate of the current state $x_d$ and counteract the effect of the delay, the estimator can propagate the obtained estimate $\hat{x}_0$ by the system dynamics \ref{eq:system} assuming zero noise to form 
\begin{equation}
		\hat{x}_d = A^d \hat{x}_0.
\end{equation}
Moreover by Assumption~\ref{as:ack} the receiver sends an acknowledgment signal to the transmitter to notify that the message was not successfully received. At the next sampling time $k=T$ the receiver only knows that the state $x_T$ due to the model \ref{eq:system} has changed to some value
\begin{equation}\label{eq:x_T}
	x_T = A^T x_0 + \sum_{m=0}^{T-1} A^{T-1-m} w_m
\end{equation}
which lies in a set given by
\begin{equation}\label{eq:set_failure}
	[A^T\ubar{X}_0 - \frac{A^T-1}{A-1} \frac{W}{2}, A^T\bar{X}_0 + \frac{A^T-1}{A-1} \frac{W}{2}],
\end{equation}
summing up the magnitude of the unknown values of the system noise $w_0, \ldots, w_{T-1}$. Then the process repeats, i.e., the sensor quantizes uniformly this new set, it sends the bin in which the state $x_T$ belongs, etc.

\underline{Case II:} Alternatively, if the message corresponding to the sample $x_0$ is successfully decoded at time $d$, the remote estimator learns in which of the $2^r$ bins the state $x_0$ belongs, e.g., $x_0 \in [ \ubar{X}_0' , \bar{X}_0' ]$ with length $\bar{X}_0' - \ubar{X}_0' =(\bar{X}_0 - \ubar{X}_0)/{2^r}$. It can construct an estimate as the center of that interval $\hat{x}_0 = (\ubar{X}_0' + \bar{X}_0')/2$.
As in the previous case the estimator can propagate the initial estimate $\hat{x}_0$ via the system dynamics to obtain a current estimate of the form $\hat{x}_d = A^d \hat{x}_0$. 
Moreover the receiver sends an acknowledgment signal to the transmitter to notify that the message was successfully received. At the next sampling time $T$ the receiver knows that the state $x_T$ given by \ref{eq:x_T} lies in a new set 
\begin{equation}\label{eq:coding_end}
	[A^T\ubar{X}_0' - \frac{A^T-1}{A-1}\frac{W}{2}, A^T\bar{X}_0' + \frac{A^T-1}{A-1} \frac{W}{2}],
\end{equation} 
because it knows that $x_0 \in [ \ubar{X}_0' , \bar{X}_0' ]$ -- note also that the set \ref{eq:coding_end} in this case is smaller than the set in the opposite case in \ref{eq:set_failure}. Then the process repeats. An illustration of this sequential coding scheme is shown in Figure~\ref{fig:example}.

\begin{figure}[t!]
	\centering
	\includegraphics[width=0.5\textwidth]{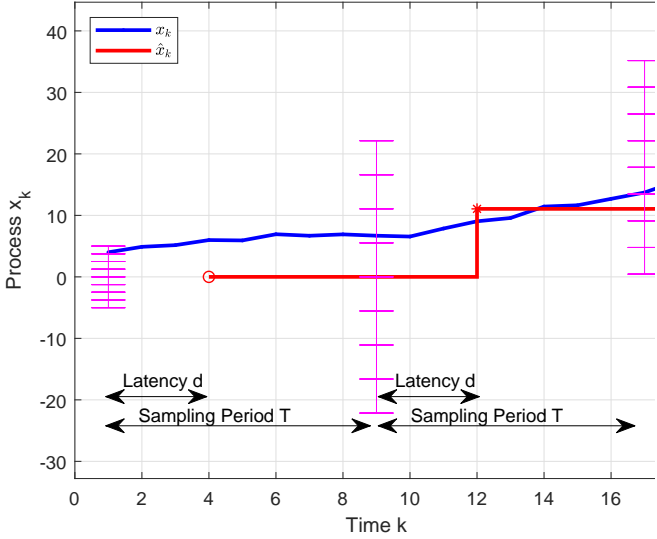}
	\caption{The process $x_k$ is sampled periodically at time steps $k=1,9,17,...$ and quantized to the 8 ($r=3$) bins shown. After a latency of $d=3$ time steps the receiver either decodes or discards the message. The first message at time $k=4$ is discarded (shown with an circle), resulting in a larger uncertainty about the state. The second message at time $k=12$ is correctly decoded (shown with a star) and the estimator obtains a good estimate.}
	\label{fig:example}
\end{figure}

Given the communication uncertainty described in the two cases above, we can derive the performance of the estimator at time $d$ as 
\begin{equation}\label{eq:estimation_error_one_step}
	\mathbb{E} | x_d - \hat{x_d} | \leq \frac{1}{2}\left[
	A^d (p_e + \frac{1-p_e}{2^r} ) (\bar{X}_0 - \ubar{X}_0) + \frac{A^d-1}{A-1} W \right]
\end{equation}
Here the expectation accounts for the uncertainty in the success of decoding. We observe in particular that the estimation error increases linearly as the decoding error probability increases. On the contrary the estimation error increases exponentially as the latency (delay) increases, hence latency has a detrimental effect on the performance motivating our analysis. We also point out that this bound depends on both the initial uncertainty about the state as well as the level of disturbance.

To get some further intuition about this estimation performance we plot in Figure~\ref{fig:one_shot_latency_reliability} the right hand side of \ref{eq:estimation_error_one_step} as a function of the latency and reliability parameters, assuming they can be arbitrarily selected, while the quantization rate $r$ is fixed. It is not surprising that the estimation performance is best under ideal conditions where the message is transmitted instantaneously without delay and with perfect reliability. However, as we discuss in Section~\ref{sec:coding} this is not practically possible as these two parameters are dependent. More importantly, from Figure~\ref{fig:one_shot_latency_reliability} we observe that there is a tradeoff between latency and reliability and this tradeoff depends on the system dynamics. We can achieve the same level of estimation performance either by increasing reliability or by decreasing latency. In particular it can be seen that if we fix the quantization level $r$, we can keep the estimation error lower than a desired value as long as the channel error behaves as $p_e = O( A^{-d})$.

\begin{figure}[t!]
	\centering
	\includegraphics[width=0.5\textwidth]{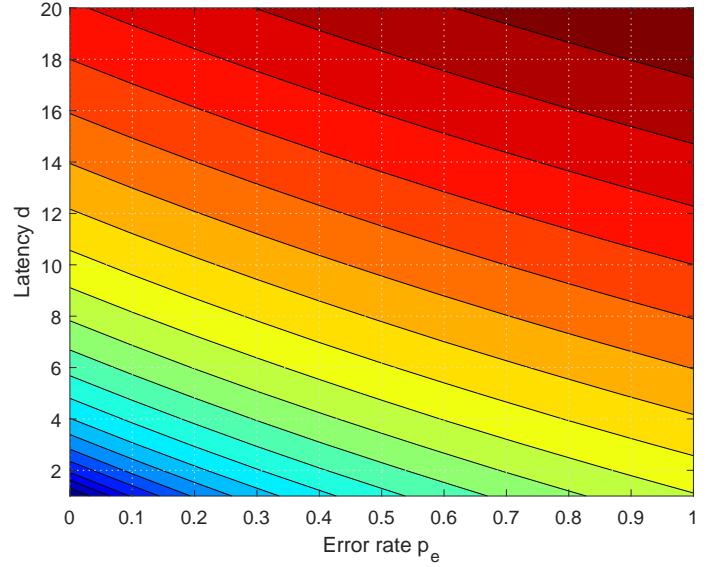}
	\caption{Sublevel sets of the single-shot estimation error \ref{eq:estimation_error_one_step} as a function of reliability and latency. Estimation is best at the minimum latency and minimum channel error. Otherwise there is a latency-reliability tradeoff that achieves the same level of estimation performance.}
	\label{fig:one_shot_latency_reliability}
\end{figure}

%%%%%%%%%%%%%%%%%%%%%%%%%%%%%%%%%%%%%%
%%%%%%%%%%%%%%%%%%%%%%%%%%%%%%%%%%%%%%
\subsection{Steady state estimation performance}
%%%%%%%%%%%%%%%%%%%%%%%%%%%%%%%%%%%%%%
%%%%%%%%%%%%%%%%%%%%%%%%%%%%%%%%%%%%%%

Since the state changes over time by \ref{eq:system}, we are interested in whether the remote estimator can track this process, and furthermore how well it can do so in the long run. Our first result derives conditions on the abstraction parameters under which estimation is feasible.

\begin{thm}[Scalar Stability Condition]\label{thm:stability}
Consider the remote estimation of the scalar dynamical system in \ref{eq:system} over the communication channel and let Assumptions~\ref{as:model} and~\ref{as:ack} hold. Suppose we employ the sequential quantization scheme described in \ref{eq:coding_start}-\ref{eq:coding_end}. Then the expected value of the estimation error $|x_{\ell T+d} - \hat{x}_{\ell T + d}|$, $\ell = 0, 1, \ldots$ is bounded if and only if
\begin{align}\label{eq:stability_condition}
	&(p_e  + \frac{1-p_e}{2^r}) A^T < 1. %\text{ and } d \leq T
\end{align}
\end{thm}

\begin{proof}
	As mentioned above for the first message there are two cases depending on the success of decoding the message. In the first case where the transmission is unsuccessful the remote estimator learns that the state $x_0$ belongs in the set $[\ubar{X}_0, \bar{X}_0]$ with width $[\bar{X}_0 - \ubar{X}_0]$ and constructs the estimate $\hat{x}_0 = (\bar{X}_0 + \ubar{X}_0 )/2$. In this case at time $k=d$ the transmitter has an estimation error about $x_0$ bounded by
	\begin{equation}\label{eq:error_bound_fail}
	|x_0 - \hat{x}_0| \leq  \frac{1}{2} (\bar{X}_0 - \ubar{X}_0 )
	\end{equation}

	In the opposite case where the remote estimator receives the message correctly, it learns that the state $x_0$ belongs in a set  $x_0 \in [ \ubar{X}_0' , \bar{X}_0' ]$ with length $\bar{X}_0' - \ubar{X}_0' =(\bar{X}_0 - \ubar{X}_0){2^r}$.  It constructs an estimate as the center of that interval $\hat{x}_0 = (\ubar{X}_0' + \bar{X}_0')/2$ with error bounded by
	\begin{equation}\label{eq:error_bound_success}
	|x_0 - \hat{x}_0| \leq  \frac{1}{2} (\bar{X}_0' - \ubar{X}_0' ) = \frac{1}{2}\frac{\bar{X}_0 - \ubar{X}_0}{2^r} 
	\end{equation}

	Let us define a random variable $\Delta_0$ as the length of the set concerning $x_0$ at the estimator at time $k=d$. That is
	\begin{equation}		
			\Delta_0:= \left\{\begin{array}{ll}
			\bar{X}_0 - \ubar{X}_0 & \text{w.p. } p_e \\
			\frac{1}{2^r}(\bar{X}_0 - \ubar{X}_0) & \text{w.p. } 1 -p_e 
			\end{array}
			\right.
	\end{equation}
	
	From the inequalities \ref{eq:error_bound_fail} and \ref{eq:error_bound_success} we also have that the estimation error is in both cases bounded by
	\begin{equation}
	|x_0 - \hat{x}_0| \leq  \frac{1}{2}  \Delta_0
	\end{equation}

	Moreover using this notation, given the success or failure of the first message, at the next sampling time $T$ the state is known to lie in a set of magnitude $A^T\Delta_0 + \frac{A^T-1}{A-1} W$. This is quantized in $2^r$ bins and sent again.
	
	Similarly let us define $\Delta_T$ as the width of the set where the remote estimator knows that the state $x_T$ belongs depending on the success or failure of the message at time $k =T+d$. This is a new random variable defined as
	\begin{equation}		
	\Delta_T:= \left\{ \begin{array}{ll}
	A^T \Delta_{0} + \frac{A^T-1}{A-1} W &\text{w. prob. } p_e \\
	\frac{1}{2^r}(A^T\Delta_{0} + \frac{A^T-1}{A-1} W ) & \text{w. prob. } 1-p_e
	\end{array}\right.
	\end{equation}
	Again we have that the new estimate satisfies $|x_T - \hat{x}_T| \leq  1/2 \Delta_T$.
	
	The process repeats so that we have the recursion at the $(\ell+1)$-th transmission
	\begin{equation}
	\Delta_{(\ell+1) T} := \left\{ \begin{array}{ll}
	A^T \Delta_{\ell T} + \frac{A^T-1}{A-1} W &\text{w. prob. } p_e \\
	\frac{1}{2^r}(A^T\Delta_{\ell T} + \frac{A^T-1}{A-1} W ) & \text{w. prob. } 1-p_e
	\end{array}\right.
	%p_e (A^T \Delta_{\ell T} + \frac{A^T-1}{A-1} W) 
	%+ (1-p_e) \frac{1}{2^r}(A^T\Delta_{\ell T} + \frac{A^T-1}{A-1} W )
	\end{equation}
	and moreover we have that
	\begin{equation}\label{eq:error_bound_delta}
	| x_{(\ell+1) T}  - \hat{x}_{(\ell+1) T}  | \leq \frac{1}{2}  \Delta_{(\ell+1) T} 
	\end{equation}
	Taking expectation at both sides we can bound the expected estimation error by
	\begin{equation}
		\mathbb{E} | x_{(\ell+1) T}  - \hat{x}_{(\ell+1) T}  | \leq \frac{1}{2}  \mathbb{E} \Delta_{(\ell+1) T}.
	\end{equation}
	Also taking the expectation in the above recursion we have that
	\begin{align}\label{eq:recursion}
	\mathbb{E}\Delta_{(\ell+1) T} 
	=
	&(p_e 	+ (1-p_e) \frac{1}{2^r}) A^T \mathbb{E}\Delta_{\ell T} 
	\notag \\
	&+  (p_e 	+ (1-p_e) \frac{1}{2^r}) \frac{A^T-1}{A-1} W
	\end{align}
	This is a linear system of equations that converges to a finite value if and only if \ref{eq:stability_condition} holds. Hence the expected estimation error is bounded if and only if \ref{eq:stability_condition}  holds.
\end{proof}

From this theorem we observe that to maintain bounded estimation there is a critical dependence between reliability, quantization level, and system dynamics. Assuming dynamics $A$ and quantization level $r$ fixed, there is a critical threshold on the communication error rate $p_e$ above which communication becomes too noisy and it is impossible to maintain a finite estimation error. Assuming dynamics $A$ and communication error rate $p_e$ fixed, there is a critical threshold on the quantization level $r$ below which information becomes too coarse and it is impossible to maintain a finite estimation error.
{Related stability results are also known -- see Remark~\ref{rem:stability}.
}

On the other hand latency $d$ does not play a role for stability as it does not appear in \ref{eq:stability_condition} - apart from the practical condition that latency should be smaller than sampling period ($d \leq T$ in Assumption~\ref{as:model}). {The fact that latency does not affect stability is also mentioned in \cite[Section~IV.F]{sahai2006necessity} and \cite{Nair_overview} based on an input-output channel model where information from the input is always delayed by a constant amount of time at the output.} We will show now however that it does play a significant role for the steady state estimation performance.

\begin{figure*}[t!]
	\centering
	\includegraphics[width=0.9\textwidth]{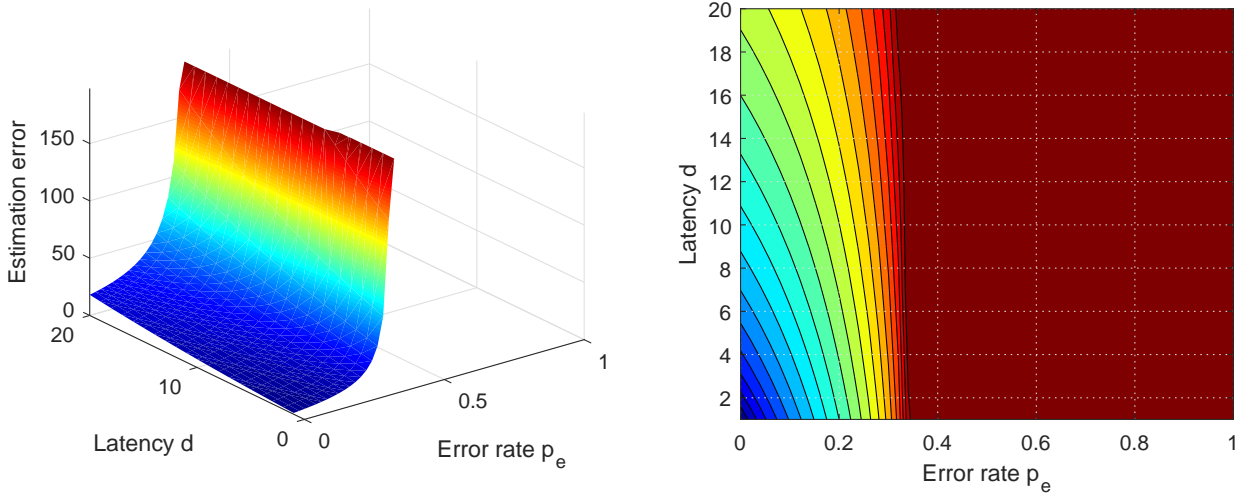}
	\caption{Sublevel sets of the steady state estimation performance \ref{eq:estimation_error_steady_state} as a function of reliability and latency. Estimation error becomes unbounded (unstable) when the reliability exceed some threshold, while larger latency adversely affects performance.}
	\label{fig:steady_state_latency_reliability}
\end{figure*}

\begin{thm}[Scalar Estimation Performance]\label{thm:error_bound}
Consider the remote estimation of the scalar dynamical system in \ref{eq:system} over the communication channel and let Assumptions~\ref{as:model} and~\ref{as:ack} hold. Suppose we employ the sequential quantization scheme described in \ref{eq:coding_start}-\ref{eq:coding_end}. If \ref{eq:stability_condition} holds, then the bound on the expected steady state error converges to the value 
\footnote{For the marginally stable case $A=1$ the bound becomes $\frac{d + (p_e  + (1-p_e)2^{-r})(T-d)}{(1-p_e)(1-2^{-r})}\frac{W}{2}$.}
\begin{align}
\limsup_{\ell \rightarrow \infty} &\,\,\mathbb{E} |x_{\ell T +d} - \hat{x}_{\ell T +d}|  \nonumber \\
&
\leq \frac{1}{2} \frac{ A^d +  (p_e  + \frac{1-p_e}{2^r})(A^T - A^d) -1}{1 - (p_e  + \frac{1-p_e}{2^r}) A^T}\frac{W}{A-1}   \label{eq:estimation_error_steady_state}
\end{align}
\end{thm}

\begin{proof}
	
	In our scheme, at each reception time $\ell T +d$ for $\ell =0,1,\ldots$ the estimator constructs an estimate $\hat{x}_{\ell T}$ about the state value $x_{\ell T}$ and then, to counteract the effect of delay, this estimate is propagated through the system dynamics \ref{eq:system} to obtain an estimate about the current state  $x_{\ell T+d}$ of the form
	\begin{equation}
	\hat{x}_{\ell T+d} = A^d \hat{x}_{\ell T}
	\end{equation}
	The current estimation error then is bounded by
	\begin{equation}\label{eq:current_error_bound}
	| x_{\ell T+d} - \hat{x}_{\ell T+d} | \leq A^d |x_{\ell T} - \hat{x}_{\ell T}| + \frac{A^d-1}{A-1} \frac{W}{2}
	\end{equation}
	accounting for the noise of the system \ref{eq:system} during this delay.

	Similar to the proof of Theorem~\ref{thm:stability}, we define define $\Delta_{\ell T}$ as the width of the set where the remote estimator knows that the state $x_{\ell T}$ belongs depending on the success or failure of the message at time $k =\ell T+d$. Moreover we have shown in \ref{eq:error_bound_delta} that  $ |x_{\ell T} - \hat{x}_{\ell T}| \leq \Delta_{\ell T}/2$ holds always. Using this fact and taking expectation of \ref{eq:current_error_bound} we get that
	\begin{equation}\label{eq:current_error_bound2}
	\mathbb{E} | x_{\ell T+d} - \hat{x}_{\ell T+d} | \leq A^d \frac{\mathbb{E}\Delta_{\ell T}}{2} + \frac{A^d-1}{A-1} \frac{W}{2}
	\end{equation}

	Recall also that $\mathbb{E}\Delta_{\ell T}$ satisfies the recursion in \ref{eq:recursion}. Hence if condition \ref{eq:stability_condition} holds then the value $\mathbb{E}\Delta_{\ell T} $ converges to
	\begin{equation}
	\lim_{\ell \rightarrow \infty} \mathbb{E} \Delta_{\ell T} = 
	\frac{ (p_e  + \frac{1-p_e}{2^r})\frac{A^T-1}{A-1} W}
	{1 - (p_e  + \frac{1-p_e}{2^r}) A^T}
	\end{equation}
	Plugging this limit expression in \ref{eq:current_error_bound2} we obtain the result \ref{eq:estimation_error_steady_state}.
\end{proof}

From \ref{eq:estimation_error_steady_state} it can be seen that, unlike the stability analysis, latency $d$ has a crucial effect on performance (due to the term $A^d$). We plot the bound \ref{eq:estimation_error_steady_state} in Figure~\ref{fig:steady_state_latency_reliability} for different values of the latency parameter $d$ and communication reliability $p_e$ and for a fixed quantization rate $r$. We observe that this steady state case differs from the finite time case of Figure~\ref{fig:one_shot_latency_reliability}. As mentioned above in \ref{eq:stability_condition} there is a minimum reliability below which the remote estimator cannot track the system and the estimation error grows to infinity. Otherwise the estimation performance improves as the error rate becomes smaller or the latency decreases.
{ We also note that a similar tradeoff between data rate and delay and with perfect reliability ($p_e=0$) appears in \cite[Sec.~II]{Nair_overview}, but without the coding design aspects of our paper.}

As already mentioned the parameters $r,d,p_e$ of our abstraction are interdependent variables of the employed channel coding scheme. In the following section we discuss how these dependencies arise from practical coding schemes, and how we can tune over these parameters to obtain the latency-reliability tradeoff that optimizes the steady state system performance.

\begin{rem}\label{rem:stability}
	{Related stability results are also known, e.g., in \cite[Prop.~4.2]{tatikonda2004control} for the case without disturbances, in \cite[Sec.~4.3]{franceschetti2014elements} for the case with disturbances of potentially unbounded support, 
		and in~\cite{Sinopoli_intermittent} for the packet-drop channel without quantization ($r\rightarrow\infty$).}
	{We point out that our stability result concerns the specific quantization scheme employed. One may wonder if it is possible to support systems with even larger eigenvalues (faster systems) by using other quantization or coding schemes. If we ignore latency and consider our communication block abstraction as an input-output channel with input $r$ bits and output that is erased with probability $p_e$, and assuming acknowledgments, then \cite[Sec.~1.4.3]{franceschetti2014elements} uses the notion of anytime capacity introduced in~\cite{sahai2006necessity} and argues that \ref{eq:stability_condition} is in fact necessary and sufficient for stability. We note however that, as detailed in the following section, our communication abstraction already includes a communication channel and corresponding coding and decoding blocks within, so the above discussion on necessary and sufficient conditions for stability may require further investigation. 
	}
\end{rem}

%%%%%%%%%%%%%%%%%%%%%%%%%%%%%%%%%%%%%%
%%%%%%%%%%%%%%%%%%%%%%%%%%%%%%%%%%%%%%
\section{Code Length Selection for Estimation Performance}\label{sec:coding}
%%%%%%%%%%%%%%%%%%%%%%%%%%%%%%%%%%%%%%
%%%%%%%%%%%%%%%%%%%%%%%%%%%%%%%%%%%%%%
In this section we investigate how channel coding can play a role in improving the system performance. That is, we use error correcting codes to alleviate the effect of channel noise on the transmitted messages. While coding leads to a dramatic increase in reliability, it causes an extra penalty in latency as reliability is obtained at the cost of sending longer messages (codewords). We will study latency-reliability tradeoffs obtained from codes and their consequences regarding the system performance. We also discuss cross-layer coding design, i.e., the selection of code parameters that yield the optimal performance taking into account the system dynamics.

\textbf{Uncoded Transmission.} As a warm-up, we first consider uncoded transmission, i.e. each time we send $r$ (information) bits through the channel uncoded and without using any error-correction mechanism.  Perhaps the simplest type of channel we could consider is the so-called \textit{binary erasure channel} with erasure probability $\zeta$ (BEC($\zeta$)). On this channel, a bit is either passed through perfectly with probability $1-\zeta$ or completely erased with probability $\zeta$. Hence, each $r$-bit message is either successfully received with probability $(1-\zeta)^r$, or otherwise with the complementary probability at least one bit is erased and the message is discarded. This corresponds to a model with $p_e = 1 - (1-\zeta)^r$.

We also model the latency $d$ of transmitting the $r$-bit message. As mentioned in Section~\ref{sec:setup} we assume the transmission of each bit requires a normalized unit time interval, hence $d = r$. With these expressions in place our communication scheme is parameterized by the single parameter $r$, the quantization level, taking values in $[1,T]$, as there are at most $T$ available time slots to transmit bits until the next sampling time. The other parameters are $\zeta$, the quality of the channel which is assumed given and not a design variable.

In this case, we can specialize the stability condition of Theorem~\ref{thm:stability} to: $\left[ 1 - (1-\zeta)^r + \left(\frac{1-\zeta}{2}\right)^r\right]A^T <1$. 
We can further plot in Figure~\ref{fig:no_coding} the optimal packet length $r$ that minimizes the steady state estimation error in \ref{eq:estimation_error_steady_state} for different values of the channel quality $\zeta$ and dynamics $A$. We observe qualitatively that longer packet lengths, and hence larger delays, are beneficial when the dynamics are fast or when the channel quality improves. Otherwise for sufficiently slow dynamics and noisy channel conditions it is better to employ smaller packet length, even a single bit, but more reliably. Estimation becomes unstable when the system dynamics become too fast or when the channel erasure becomes more likely.

\begin{figure}[t!]
	\centering
	\includegraphics[width=0.5\textwidth]{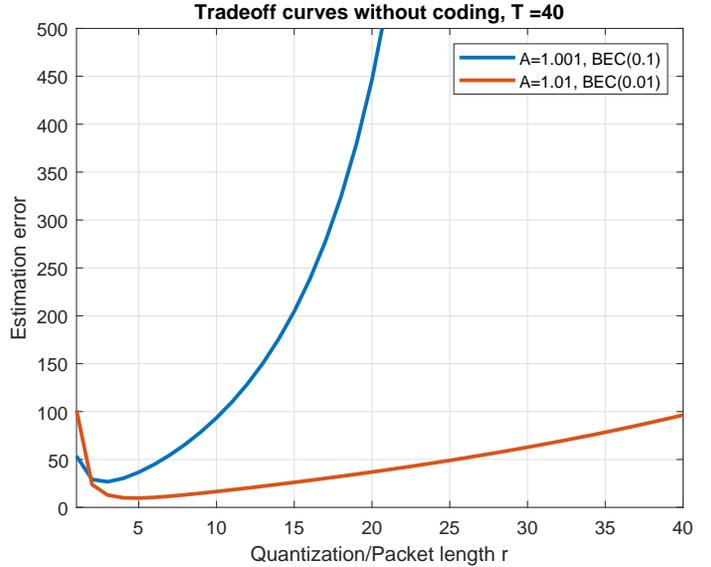}
	\caption{Estimation performance of uncoded communication over the binary erasure channel. We observe that there is an optimal quantization level that achieves the best latency-reliability tradeoff. Longer packet lengths, and hence larger delays, are beneficial when the dynamics are fast or when the channel quality improves. }
	\label{fig:no_coding}
\end{figure}

\textbf{Coded Transmission (using finite blocklength codes).} 
We consider error-correction coding systems which encode the $r$-bit information messages into longer $n$-bit codewords. Here we are assuming that the transmission takes place over a channel with binary input (e.g. the binary erasure channel). Clearly, the rate of the code needs to be lower than the channel capacity $\frac{r}{n} < C$. We attribute the latency $d$ of communicating the $n$-bit message in transmission delay of sending these $n$ bits - in practice there will also be other overheads contributing to delay that we omit from the present analysis. For example we model $d = n$, i.e., each bit requires a single (normalized) time unit. 
The expected reliability (error probability) of the code is denoted by $p_e$, i.e. with probability $1-p_e$ we can decode the information message correctly (from the noisy outcome of the channel) and with probability $p_e$ the decoding procedure is unsuccessful. As mentioned in Assumptions~\ref{as:model},~\ref{as:ack} we further assume that the decoder can \textit{detect} whether or not the decoding procedure has been successful, and will notify the transmitter about it by using a one bit ACK/NACK feedback -- see also Remark~\ref{rem:detection}.
As a result, in case of decoding failure, the transmitted packet will be discarded.   We can thus attribute the channel reliability $1 - p_e$ to the probability of successful decoding of the coded messages. Using Theorem~\ref{thm:error_bound}, the expected steady state error of the overall system is given as 
	\begin{align}
& \lim_{\ell \rightarrow \infty} \mathbb{E} |x_{\ell T +d} - \hat{x}_{\ell T +d}|  \nonumber \\
& \,\,\,\, \,\,\,\,= \frac{ A^{n} +  (p_e  + \frac{1-p_e}{2^r})(A^T - A^n) -1}{1 - (p_e  + \frac{1-p_e}{2^r}) A^T}\frac{W}{A-1}. \label{eq:steady_state_coding}
	\end{align}
From this expression it is apparent that the relation between code parameters $r,n,p_e$ plays a key role in characterizing the trade-off between the steady-state error and latency of the system. 
Fundamental information--theoretic laws \cite{polyanskiy2010channel,dobrushin1961mathematical, strassen} state that to reliably communicate over a channel with capacity $C$, the \emph{optimal} (shortest possible) blocklength
$n$ scales as 
\begin{align}
 & n  \approx \frac{V Q^{-1}(p_e)}{(C-R)^2},  \quad {\text{ or equivalently: }}  \nonumber\\
 &  p_e = Q\left(\sqrt{\frac{n}{V}}(C-R)  + O( \log n ) \right) ,  \label{opt_scaling}
\end{align}
where $Q(\cdot)$ is the tail probability of the standard normal distribution, $V$ is a characteristic of the channel referred to as channel dispersion, and $R = r/n$.
But the only optimal codes (in the sense of \ref{opt_scaling}) we currently know have an exponential decoding complexity in general. Crucially,  we  quest for codes with \emph{low-complexity}  encoder/decoder design (ideally, linear in blocklength). For the error-correcting codes used in practice, one can compute (simulate) the \textit{trade-off} curves showing how the codelength $n$ varies with the rate $R = r / n$ at different values of error probability $p_e$. By plugging such trade-off curves into \ref{eq:steady_state_coding} we can obtain trade-off curves for reliability versus latency.  

\begin{figure}[t!]
	\centering
	\includegraphics[width=0.5\textwidth]{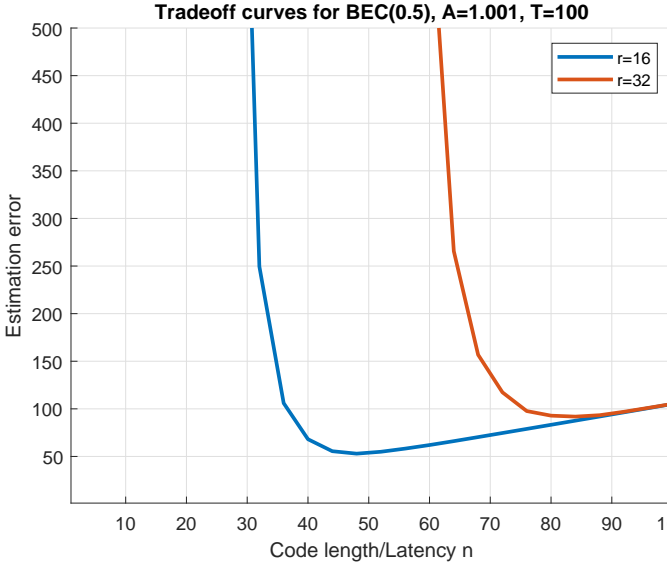}
	\caption{Estimation performance of different random coding schemes over the binary erasure channel with erasure probability $0.5$. We observe that there is an optimal codelength that achieves the best latency-reliability tradeoff. For $r=16$ quantization bits the optimal codelength is $n\approx45$ bits, while for $r=32$ the optimal codelength is $n\approx85$ bits.}
	\label{fig:random_coding}
\end{figure}

Let us now consider again the \textit{binary erasure channel (BEC)} as the transmission medium and use \textit{random linear codes} for error correction. Such codes are constructed by selecting uniformly at random a linear mapping from the space $\{0,1\}^r$ to $\{0,1\}^n$. Such a mapping is specified by a $k\times n$ binary matrix whose entries are chosen i.i.d according to Bernoulli($\frac12$). Every such random mapping (matrix) is a code that can be use to encode an $r$-bit information vector to a codeword of size $n$.  For the specific case of BEC, random codes can be decoded in cubic time in terms of the length and are known to have the optimal rate-reliability-length as in \ref{opt_scaling} (see e.g. \cite{amraoui2009finite}). In other words, such codes minimize the latency of the communication block of the system in Figure~\ref{fig:channel_model}. However, as their encoding/decoding complexity is rather high (cubic for BEC, exponential for other channels) we seldom use these codes in practice and other codes with linear complexity and higher latency (such as iterative) codes are preferred. 
For given system dynamics $A$ and sampling period $T$ we consider different code choices given by the quantization level $r$ (which indicates the number of raw information bits) and codelength $n$. We obtain the error rate of such coding by simulation, and we plot in Figure~\ref{fig:random_coding} the resulting estimation error by~\ref{eq:steady_state_coding}. We observe that there is an optimal codelength that minimizes estimation error. Shorter codes increase estimation error because they have lower reliability. The difference in performance can be significant, for example, in the figure, there is 50\% improvement in estimation quality between the optimal length and the longest length. Alternatively longer codes have better reliability but with resulting latency at the expense of estimation performance. As a side note, we point out that for the considered channel and dynamics values the preceding case with uncoded messages leads to instability.

A similar plot for stable system dynamics is shown in Figure~\ref{fig:stable}. Unbounded estimation error (unstable) cannot occur in this case because even if the probability of error is $p_e=1$ the estimation error is finite. But a performance tradeoff in codelength exists. We see that an appropriate codelength selection can decrease the estimation error by up to $30\%$ compared to longer codes.

\begin{figure}[t!]
	\centering
	\includegraphics[width=0.5\textwidth]{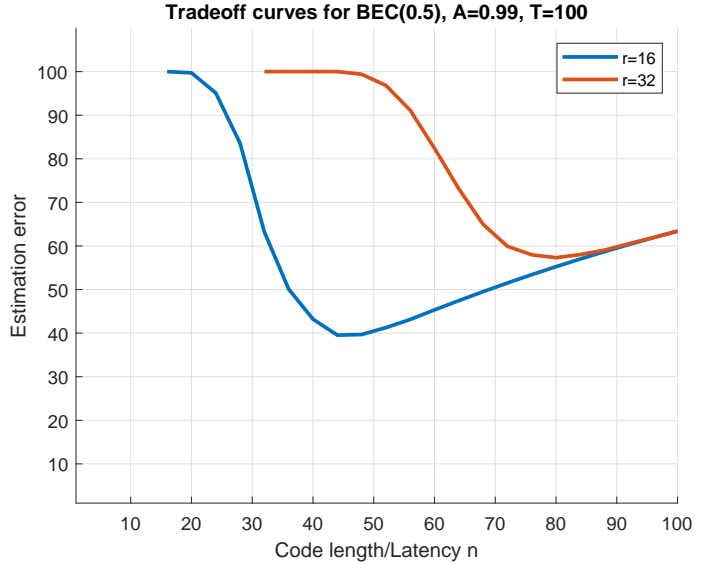}
	\caption{Estimation performance of different random coding schemes for a stable system over the binary erasure channel with erasure probability $0.5$. We observe that there is an optimal codelength that achieves the best latency-reliability tradeoff.}
	\label{fig:stable}
\end{figure}

Carrying on similar plots we also obtain the following \textit{practical insights}: for faster dynamics (larger $A^T$) and noisy channel conditions (larger probability of erasure $\zeta$) longer codelengths are preferred. This is expected as faster systems require higher reliability and noisy channels longer codes with more redundancy. On the other hand, for slower dynamics and less noisy channels shorter codelengths are optimal, i.e., there is no performance gain from longer codes. In other cases the codelengths should be neither too short nor too long.

\textbf{Theoretically Optimal Blocklength.} In principle, we can use equations \ref{eq:steady_state_coding} and \ref{opt_scaling} to find the optimal delay that minimizes the steady state error in terms of the number of information bits $r$. This can be done by deriving equation  \ref{eq:steady_state_coding}  in terms of $n$ and noting that $p_e$ is implicitly dependent with $n$ through \ref{opt_scaling}. However, the resulting equation does not attain a simple closed form solution and hence we need to resort to approximate solutions. Nevertheless,  the equations can be solved numerically by bisection and the optimal delay $n$ can be found in terms of $r$ and the parameters of the dynamical system and the channel. 

Let us now provide a heuristic argument to approximate the optimal delay $n$ in the regime where the value of $A$ is close to $1$, e.g. $A = 1.1$ or smaller. 
Let $\theta(n) = p_e(n)  + \frac{1-p_e(n)}{2^r}$. Then taking the derivative of \ref{eq:steady_state_coding} with respect to $n$ and equaling to zero, we obtain that the optimal codelength satisfies
\begin{equation}\label{eq:optimal_length1}
	\log(A)(1-\theta(n))(1-\theta(n)A^T) + \frac{d\theta}{dn} (A^T-1) =  0
\end{equation}
where $\log$ is the natural logarithm. Suppose $\theta(n)$ is small enough so that we may approximate $(1-\theta(n))(1-\theta(n)A^T)\approx 1-\theta(n)(A^T+1)$. Further suppose $A\approx 1+ a$ for some small $a>0$ so that $\log(A) \approx a$ and $A^T \approx  1+ aT$. Using these approximations in \ref{eq:optimal_length1} we get
\begin{equation}\label{eq:optimal_length2}
1-\theta(n)(2 + a T) + \frac{d\theta}{dn} T \approx 0
\end{equation}

Moreover let us approximate  $\theta(n) \approx p_e(n)$ which is reasonable for a large number of information bits $r$. From \ref{opt_scaling} let us approximate the normal tail as $Q(x) \approx \phi(x)/x$ where $\phi(x)$ is the standard normal density function. We also have that $\frac{dQ(x)}{dx} = -\phi(x)$.
Let us also assume the code rate $R=r/n$ in \ref{opt_scaling} is fixed to be independent of $n$ then we obtain from \ref{eq:optimal_length2} that the optimal code length satisfies
\begin{align}
1-\frac{\phi(\sqrt{\frac{n}{V}}(C-R))}{\sqrt{\frac{n}{V}}(C-R)}&(2 + a T) \notag\\
- &\phi(\sqrt{\frac{n}{V}}(C-R))\frac{C-R}{2 \sqrt{Vn}} T \approx 0
\end{align}
or equivalently
%%
%\begin{align}
%\phi(\sqrt{\frac{n}{V}}(C-R)) \approx \frac{\sqrt{\frac{n}{V}}(C-R)}{2 + aT + 1/2 \frac{(C-R)^2}{V} T}
%\end{align}
%%
%or
%
\begin{align}
\frac{\phi(\sqrt{\frac{n}{V}}(C-\frac{r}{n}))}{\sqrt{\frac{n}{V}}(C-\frac{r}{n})} 
\approx \left[ 2 + aT + 1/2 \frac{(C-\frac{r}{n})^2}{V} T\right]^{-1}
\end{align}
%
%%
%\begin{align}
%Q(\sqrt{\frac{n}{V}}(C-R)) \approx \frac{1}{2 + aT + 1/2 \frac{(C-R)^2}{V} T}
%\end{align}
%%
This equation does not have a closed form solution with respect to $n$. From this equation we see that both increasing the system eigenvalue $a$ and increasing the sampling period $T$ have the effect of increasing the optimal codelength $n$. However the effect of the sampling period is more significant at code rates further from the channel capacity ($R<<C$). Numerically we also see that this approximation compares well with the optimal codelength shown in Figure~\ref{fig:random_coding}, providing $n\approx37$ and $n\approx70$ for the cases $r=16$ and $r=32$ respectively.

\begin{rem}\label{rem:detection}
	
	The case where the channel may corrupt the transmitted message so that the receiver is not certain about what message was sent or may incorrectly decode a message is a challenging problem in control -- see for example recent work~\cite{sukhavasi2016linear} -- and is not further explored in this paper. 
	We note that when transmission is over the BEC channel, almost all the code designs in practice are capable of detecting decoding failures~\cite{richardson2008modern}. For other channels, the task of detecting the decoding failures is typically facilitated by using the use of cyclic redundancy check codes. Adding a cyclic redundancy check code on top of e.g., the channel codes discussed in Section~\ref{sec:coding}, can make the probability of incorrect decoding negligible. 
\end{rem}

\section{Latency-Reliability Tradeoffs for Higher-dimensional systems}\label{sec:multidimensional}

In this section we consider higher dimensional systems. Our goal is to fix a quantization scheme and develop a methodology to analyze the rate-latency-reliability tradeoff in state estimation. In particular this also includes the case where a system may have both unstable and stable modes.

Consider the system
\begin{equation}\label{eq:system_multidimensional}
x_{k+1} = A x_k + w_k
\end{equation}
where $x_k \in \mathbb{R}^n$, and we denote each element as $x_k(i), i=1, \ldots,n$, and $A\in \mathbb{R}^{n \times n}$ and $w_k$ is bounded noise with $w_k(i) \in [-W(i)/2, W(i)/2]$.

Tatikonda-Mitter~\cite{Tatikonda} propose a scheme that goes as follows. We can perform a similarity transform to obtain a system state evolution in the normal Jordan form. In the new coordinate system the receiver/estimator keeps a hyper-cube where the state belongs. The sensor/transmitter also keeps track of this hypercube and quantizes uniformly each element of this hypercube with a different number of bits per dimension. At the next sampling time, the sensor updates each coordinate of the hypercube to obtain a new hypercube where the new state sample is guaranteed to belong. The estimation error eventually accounts for the quantization performed along each dimension and along all Jordan blocks.

We explain this scheme in more detail in our setup. As the sensor in our setup is sampling the system every $T$ time steps this corresponds to system dynamics
\begin{equation}\label{eq:multidimensional}
x_T = A^T x_0 + \sum_{\ell=0}^{T-1} A^{T-1-\ell}  w_\ell
\end{equation}
Specifically let $A^T = \Phi^{-1} J \Phi$ be a Jordan decomposition and for simplicity of exposition assume that $A$ has only real eigenvalues\footnote{Complex eigenvalues can be handled as in \cite{Tatikonda}. This setup ends up in a slightly different coordinate state transformation and a slightly different recursion than \ref{eq:vector_recursion} and hence a different estimation error bound in \ref{eq:vector_steady_state}. The stability condition \ref{eq:vector_stability} remains the same.}. At sampling time $T$ the sensor observes the current state $x_T$ and applies the transform $z_T = \Phi x_T$. The sensor needs to quantize the transformed state. The receiver/estimator knows that the state belongs in some $n$-dimensional cube. The transmitter uniformly quantizes each dimension $i$ of this set into $2^{r_i}$ bins and transmits the $r_i$-bit message indicating the bin in which the measured state $x_T(i)$ belongs for all dimensions $i$. In our paper we take these values $r_i$ fixed and they satisfy $\sum_{i=1}^n r_i = r$, but it is possible to also optimize over how we allocate them.  After the latency of $d$ time steps, given the success or fail of the transmission, the receiver builds a state estimate $\hat{z}_T$ of the state $z_T$, applies the inverse transform $\hat{x}_T = \Phi^{-1} \hat{z}_T$ to obtain an estimate about the past state $x_T$, and propagates the dynamics to obtain an estimate $\hat{x}_{T+d} = A^d \hat{x}_T = A^d\Phi^{-1} \hat{z}_T$ to counteract the effect of latency as in Section~\ref{sec:setup}. Next we present the quantization method and the evolution of the estimation error with respect to the transformed state $z_T$, and at the end we discuss the performance with respect to the original system state $x_T$.

Consider the situation at time step $k=T$. The sensor measures $z_T$. At the same time, the remote estimator knows that the \textit{previous} sample $z_{0}$ lies in a hypercube, described by $z_0(i) \in [\ubar{Z}_0(i), \bar{Z}_0(i)]$ for each dimension $i=1, \ldots, n$. Similar to the proof of Theorem~\ref{thm:stability} let us define the random variable  $\Delta_0(i) = \bar{Z}_0(i) - \ubar{Z}_0(i)$ which denotes the width along each dimension. 

At the sampling time $k=T$ and before the transmission takes place the receiver only knows that the state $z_T$ has changed due to the model \ref{eq:multidimensional} to some value
\begin{equation}
z_T = J z_0 +  \sum_{m=0}^{T-1} \Phi A^{T-1-m} w_m
\end{equation}
where we employed the Jordan decomposition. Consider the $i$th element of this vector and suppose that it corresponds to some (real by assumption) eigenvalue $\lambda$. It takes the value
\begin{align}
z_T(i) = \lambda z_0(i) + &J_{i,i+1} z_0(i+1) \notag\\
&+ \sum_{m=0}^{T-1} \sum_{j=1}^n (\Phi A^{T-1-m})_{ij} w_m(j)
\end{align}
where we used the fact that $J$ is in the Jordan normal form and the element $J_{i,i+1} $ of this matrix is non-negative because it takes the value either 0 or 1. Given the previous information about $z_0(i) \in [\ubar{Z}_0(i), \bar{Z}_0(i)]$ then, the remote estimator, before transmission, knows that the state $z_T$ lies in a hypercube where the $i$th element of this vector is upper bounded by
\begin{align}
z_T(i) 
\leq 
&\max\{\lambda \bar{Z}_0(i), \lambda \ubar{Z}_0(i)\} + J_{i,i+1}  \bar{Z}_0(i+1) \\
&+ \sum_{m=0}^{T-1} \sum_{j=1}^n |(\Phi A^{m})_{ij} |  W(j)/2
\end{align}
where we employed the two extreme cases for $z_0(i)$ and we sum up the magnitude of the unknown values of the system noise $w_0, \ldots, w_{T-1}$. Similarly it is bounded below by
\begin{align}
z_T(i) \geq 
&\min\{\lambda \bar{Z}_0(i), \lambda \ubar{Z}_0(i)\} +  J_{i,i+1} \ubar{Z}_0(i+1) \notag\\
&- \sum_{m=0}^{T-1} \sum_{j=1}^n |(\Phi A^{m})_{ij} |  W(j)/2
\end{align}
Combining the upper and lower bounds we have that $z_T(i)$ lies in an interval of length
\begin{align}
|\lambda| (\bar{Z}_0(i)- \ubar{Z}_0(i)) &+  J_{i,i+1} (\bar{Z}_{0}(i+1) - \ubar{Z}_0(i+1)) \notag\\
&+ \sum_{m=0}^{T-1} \sum_{j=1}^n |(\Phi A^{m})_{ij} |  W(j)
\end{align}
Note that the absolute value of the eigenvalue has appeared. Using the above introduced notation this length is equal to
\begin{align}
|\lambda| \Delta_0(i) +  J_{i,i+1}\Delta_0(i+1) 
+ \sum_{m=0}^{T-1} \sum_{j=1}^n |(\Phi A^{m})_{ij} |  W(j)
\end{align}
We can further express these interval lengths along all dimensions $i=1, \ldots, n$ in a vector form 
\begin{equation}
	|J| \Delta_{0} 
	+ \sum_{m=0}^{T-1} |\Phi A^m |  W
\end{equation}
where by $|M|$ we denote a matrix whose elements are the absolute values of the elements of the matrix $M$.

The sensor quantizes uniformly each dimension $i$ of this new hypercube with a selected number of bits $r_i$, and it sends the bin in which the state belongs.

\underline{Case I:} With probability $p_e$ the message corresponding to the sample $z_T$ is not successfully received at time $k=T+d$, the remote estimator only knows the original information before transmission, i.e., that $z_T$  lies in the above hypercube with widths given by $|J| \Delta_{ 0} + \sum_{m=0}^{T-1} |\Phi A^m |  W $. The estimator selects the center of this hypercube as the estimate $\hat{z}_T$.

\underline{Case II:} Alternatively, if the message corresponding to the sample $z_T$ is successfully decoded at time $T+d$, the remote estimator learns in which of the $2^{r_1 + \ldots + r_n}$ bins the state $z_T$ belongs, it can construct an estimate as the center of that interval. In that case the estimator knows that the state lies in a hypercube with smaller widths (as compared to Case I) given by $ \text{diag}\{2^{-r_i}\}
( |J| \Delta_{0} 
+ \sum_{m=0}^{T-1} |\Phi A^m |  W )$

Let us define $\Delta_T$ as the width of the set where the remote estimator knows that the state $z_T$ belongs depending on the success or failure of the message at time $k =T+d$. This is a new random variable defined as
\begin{equation}	\label{eq:random_recursion}	
\Delta_T:= \left\{ \begin{array}{ll}
 |J| \Delta_{ 0} 
+ \sum_{m=0}^{T-1} |\Phi A^m |  W  &\text{w. p. } p_e \\
 \text{diag}\{2^{-r_i}\}
( |J| \Delta_{0} 
+ \sum_{m=0}^{T-1} |\Phi A^m |  W ) &\text{w. p. } 1-p_e
\end{array}\right.
\end{equation}
depending on the Case I or II above. By Assumption~\ref{as:ack} the receiver sends an acknowledgment signal to the transmitter to notify whether the message was successfully received or not. Then the process repeats.

%Taking expectation at both sides we can bound the expected estimation error by
%%
%\begin{equation}
%\mathbb{E} | z_{(\ell+1) T}(i)  - \hat{z}_{(\ell+1) T}(i)  | \leq \frac{1}{2}  \mathbb{E} \Delta_{(\ell+1) T}(i).
%\end{equation}
%%

At a general sampling time $\ell T$ we have that the width of the expected hypercube follows the general recursion 
\begin{align}\label{eq:vector_recursion}
\mathbb{E}\Delta_{(\ell+1) T} 
= \big(p_e I 	+ &(1-p_e) \text{diag}\{2^{-r_i}\}\big) \notag\\
&
\left( |J| \mathbb{E}\Delta_{\ell T} 
+ \sum_{\ell=0}^{T-1} |\Phi A^\ell |  W \right)
\end{align}
which follows by taking the expectation in \ref{eq:random_recursion}.

Since $|J|$ is an upper triangular matrix (due to the Jordan form), its eigenvalues are the diagonal elements, which correspond to the absolute values $|\lambda_i(A^T)|$ of the eigenvalues of matrix $A^T$ which are the same as $|\lambda_i(A)|^T$. It follows that the above recursion \ref{eq:vector_recursion} converges to a finite value if and only if the number of bits per dimension and the reliability satisfy  $(p_e + (1-p_e) \frac{1}{2^{r_i}}) |\lambda_i(A)|^T <1 $ for all $i$. This is a generalization of Theorem~\ref{thm:stability}. Similar results are shown in~\cite[Prop. 5.1]{Tatikonda}. Let $\Delta^*$ denote the vector that is the unique steady state solution of the above recursion~\ref{eq:vector_recursion}, i.e., 
\begin{align}\label{eq:vector_fixed_point}
\Delta^*
= \Bigg(p_e I 	+ &(1-p_e) \text{diag}\{2^{-r_i}\}\Bigg) 
\left( |J| \Delta^* 
+ \sum_{\ell=0}^{T-1} |\Phi A^\ell |  W \right)
\end{align}

Let us now return to the estimation error. We have by construction that for all sampling times and for all dimensions $i=1,\dots, n$ that
\begin{equation}\label{eq:zeta_delta}
| z_{\ell T}(i)  - \hat{z}_{\ell T}(i)  | \leq \frac{1}{2}  \Delta_{\ell T}(i)
\end{equation}
We can then measure the estimation error with respect to the original system dynamics as well as with respect to the current state\footnote{In this section we consider the $\ell_1$ norm where $\|x - \hat{x}\|_1 = \sum_{i=1}^n |x(i)-\hat{x}(i)|$.} by
\begin{align}
&\| x_{\ell T+d} - \hat{x}_{\ell T+d} \| \notag\\
& = \|\left. A^d \Phi^{-1} (z_{\ell T} - \hat{z}_{\ell T}) + \sum_{m=0}^{d-1} A^{d - 1 - m} w_{\ell T +m }\right.\|
\end{align}
where we used the fact that $\hat{x}_{\ell T+d} =A^d \Phi^{-1} \hat{z}_{\ell T}$ as explained in the beginning of this section. We can bound this estimation error as
\begin{align}
&\| x_{\ell T+d} - \hat{x}_{\ell T+d} \| \notag\\
&\leq \| A^d \Phi^{-1} \| \| z_{\ell T+d} - \hat{z}_{\ell T+d}\| + \sum_{m=0}^{d-1} \|A^{d - 1 - m}\|  \|w_{\ell T +m }\| 
\notag\\
&\leq \frac{1}{2} \| A^d \Phi^{-1} \| \sum_{i=1}^n \Delta_{\ell T + d}(i) + \sum_{m=0}^{d-1} \|A^{m}\| \sum_{i=1}^n \frac{1}{2} W(i)
\end{align}
where in the last inequality we used \ref{eq:zeta_delta}. Taking the expectation and the limit in the above expression we obtain the following result

%	%
%\begin{align}\label{eq:vector_steady_state}
%&\limsup_{\ell \rightarrow \infty} \mathbb{E} \| x_{\ell T+d} - \hat{x}_{\ell T+d} \| \\
%&\leq \frac{1}{2} \| A^d \Phi^{-1} \| \sum_{i=1}^n \Delta^*(i) + \sum_{m=0}^{d-1} \|A^{m}\| \sum_{i=1}^n \frac{1}{2} W(i)
%\end{align}
%%
%where we substituted the steady state solution $\Delta^*$ of the recursion \ref{eq:vector_recursion}. 

\begin{thm}[Vector Estimation Performance]\label{thm:vector_error_bound}
	Consider the remote estimation of the dynamical system in \ref{eq:system_multidimensional} over the communication channel and let Assumptions~\ref{as:model} and~\ref{as:ack} hold. Suppose we employ the sequential quantization scheme described in this section. Let $A^T = \Phi^{-1} J \Phi$ be a Jordan decomposition and suppose the system has only real eigenvalues. Moreover suppose  
	\begin{equation}\label{eq:vector_stability}
		(p_e + (1-p_e) 2^{-r_i}) |\lambda_i(A)|^T <1
	\end{equation}
	 holds for all $i=1, \ldots,n$. Then the expected steady state estimation error is bounded by
	\begin{align}\label{eq:vector_steady_state}
	\limsup_{\ell \rightarrow \infty}\,  &\mathbb{E}\| x_{\ell T+d} - \hat{x}_{\ell T+d} \| \notag\\
	&\leq \frac{1}{2} \| A^d \Phi^{-1} \| \sum_{i=1}^n \Delta^*(i) + \sum_{m=0}^{d-1} \|A^{m}\| \sum_{i=1}^n \frac{1}{2} W(i)
	\end{align}
	where $\Delta^*$ is the solution to the linear equation \ref{eq:vector_fixed_point}.
\end{thm}

To summarize, we can solve for the recursion steady state solution \ref{eq:vector_fixed_point}, plug in the above expression \ref{eq:vector_steady_state}, and obtain an expression on the estimation cost. This is a generalization of Theorem~\ref{thm:error_bound} which is in fact recovered in the scalar dynamics case. Again we see that rate $r$, reliability $p_e$ and latency $d$ have a significant effect on performance.

In the following we analyze numerically the above rate-latency-reliability tradeoff.

\subsection{Finite blocklength selection for higher-dimensional systems}

Consider the system with double integrator dynamics $A=\left[\begin{array}{cc}1 &0.1\\ 0 &1\end{array}\right]$ and driven by noise in the second state with magnitude $W=\left[\begin{array}{c}0 \\1\end{array}\right]$. This corresponds for example to tracking the position of an object subject to random acceleration. Suppose we fix different quantization levels $r$ and for simplicity we allocate an equal number of bits to quantize each dimension of this system. We consider again the Binary erasure channel with parameter $\zeta$ and optimal coding for this channel as explained in Section~\ref{sec:coding}. In Figure~\ref{fig:multidimensional} we plot the estimation performance for the first state as a function of the code length. This is computed using the expression as in \ref{eq:vector_steady_state}. Note that in this example even though the first state is not subject to noise there is still the effect of noise carried over from the second state. We observe quantitatively similar tradeoffs between shorter or longer codes as in the scalar system case of Section~\ref{sec:coding}. Longer codes have better reliability but they increase latency which in turn affects estimation, hence they should be avoided.

\begin{figure}[t!]
	\centering
	\includegraphics[width=0.5\textwidth]{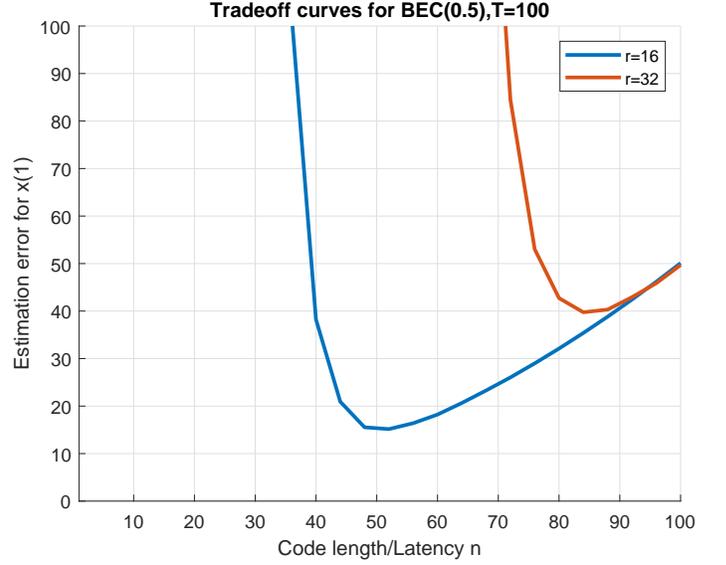}
	\caption{Estimation performance of different random coding schemes for a higher dimensional system over the binary erasure channel with erasure probability $0.5$. We observe that there is an optimal codelength that achieves the best latency-reliability tradeoff.}
	\label{fig:multidimensional}
\end{figure}

\section{Conclusion}
We consider a new communication abstraction motivated by the recent interest for low-latency high-reliability applications in the Internet-of-Things. More specifically we examine the tradeoffs between latency and reliability for the problem of estimating scalar dynamical systems over communication channels. We couple this approach with different latency-reliability curves derived from practical coding schemes. Our methodology enables a novel co-design, i.e., select the appropriate codelength depending on the system dynamics to optimize estimation performance. 

Future work is focused on extending the co-design methodology from estimation to low-latency high-reliability control of fast dynamical systems.

% Bibliography
%==================================================================%
\bibliographystyle{bib/IEEEtran}
\bibliography{bib/Intro,bib/power,bib/random_access,bib/state_aware,bib/coding,bib/Coding_control_new,bib/myIEEEabrv,bib/myIEEEfull,bib/bib-pooling,bib/ribeiro,bib/robotics,bib/low_latency_new}
%==================================================================%

\end{document}